\documentclass{amsart}
\usepackage{amsmath}
\usepackage{amssymb}

\usepackage[all]{xy}

\DeclareMathOperator{\Env}{Env} \DeclareMathOperator{\Vol}{Vol}
\DeclareMathOperator{\val}{val} \DeclareMathOperator{\dvs}{div}
\DeclareMathOperator{\Ex}{Ex} \DeclareMathOperator{\Supp}{Supp}
\DeclareMathOperator{\Div}{Div} \DeclareMathOperator{\CDiv}{CDiv}
\DeclareMathOperator{\BigProj}{\textbf{Proj}}
\DeclareMathOperator{\Spec}{Spec}

\newtheorem{theorem}{Theorem}[section]
\newtheorem{lemma}[theorem]{Lemma}
\newtheorem{proposition}[theorem]{Proposition}
\newtheorem{corollary}[theorem]{Corollary}

\theoremstyle{definition}
\newtheorem{definition}[theorem]{Definition}

\theoremstyle{remark}
\newtheorem{remark}[theorem]{Remark}

\numberwithin{equation}{section}

\begin{document}

\title{On the Volume of Isolated Singularities}

\author{Yuchen Zhang}
\address{Department of Mathematics, University of Utah, Salt Lake City, Utah 84102}
\curraddr{155 S 1400 E Room 233, Salt Lake City, Utah 84112}
\email{yzhang@math.utah.edu}



\date{\today}



\begin{abstract}
We give an equivalent definition of the local volume of an isolated
singularity $\Vol_\text{BdFF}(X,0)$ given in \cite{BdFF} in the
$\mathbb{Q}$-Gorenstein case and we generalize it to the
non-$\mathbb{Q}$-Gorenstein case. We prove that there is a positive
lower bound depending only on the dimension for the non-zero local
volume of an isolated singularity if $X$ is Gorenstein. We also give
a non-$\mathbb{Q}$-Gorenstein example with
$\Vol_\text{BdFF}(X,0)=0$, which does not allow a boundary $\Delta$
such that the pair $(X,\Delta)$ is log canonical.
\end{abstract}

\maketitle

\section{Introduction}

Let $X$ be a normal variety and let $f:X\rightarrow X$ be a finite
endomorphism, i.e. a finite surjective morphism of degree $>1$.  If
$X$ is projective, an abundant literature shows that the existence
of an endomorphism imposes strong restrictions on the global
geometry of $X$. In a recent paper \cite{BdFF}, Boucksom, de Fernex
and Favre introduce the volume $\Vol_{\text{BdFF}}(X,0)$ of an
isolated singularity and show that $\Vol_{\text{BdFF}}(X,0)=0$ if
there exists an endomorphism preserving the singularity. When $X$ is
$\mathbb{Q}$-Gorenstein, they show that $\Vol_{\text{BdFF}}(X,0)=0$
if and only if $X$ has log canonical singularities. For a better
understanding of $\Vol_{\text{BdFF}}(X,0)$, they propose two
problems:

\begin{description}
\item[Problem A]Does there exist a positive lower bound, only depending on the
dimension, for the volume of isolated Gorenstein singularities with
positive volume?

\item[Problem B] Is it true that $\Vol_{\text{BdFF}}(X,0)=0$
implies the existence of an effective $\mathbb{Q}$-boundary $\Delta$
such that the pair $(X,\Delta)$ is log-canonical? (The converse
being easily shown).
\end{description}

In this paper, we will give an alternative definition of the volume
in the $\mathbb{Q}$-Gorenstein case via log canonical modification
(the existence of these modifications is shown in \cite{OX}). By
using the DCC for the volume established in \cite{HMX12}, we will
give a positive answer to Problem A.

\begin{theorem}(See Theorem \ref{DCC})
There exists a positive lower bound, only depending on the
dimension, for the volume of isolated Gorenstein singularities with
positive volume.
\end{theorem}

We will then generalize this new definition to the
non-$\mathbb{Q}$-Gorenstein case. We will define the augmented
volume $\Vol^+(X)$ as the liminf of the $m$-th limiting volumes
$\Vol_m(X)$. We will show that the augmented volume $\Vol^+(X)$ is
greater than or equal to the volume $\Vol_{\text{BdFF}}(X,0)$. When
there exists a boundary $\Delta$ on $X$ such that the pair
$(X,\Delta)$ is log canonical, we have the following theorem

\begin{theorem}(See Corollary \ref{vol=lc})
The following statements are equivalent:
\begin{enumerate}
\item There exists a boundary $\Delta$ on $X$ such that $(X,\Delta)$
is log canonical.
\item $\Vol_m(X)=0$ for some (hence any multiple of) integer
$m\geqslant 1$.
\end{enumerate}
\end{theorem}

In Section \ref{cone_sing}, we will give a counterexample to Problem
B.

\begin{theorem}(See Theorem \ref{counterexample})
There exists a polarized smooth variety $(V,H)$ such that the affine
cone $X=C(V,H)$ has positive $\Vol_m(X)$ for any positive integer
$m$, but $\Vol_{\text{BdFF}}(X,0)=0$.
\end{theorem}

A new idea is needed to investigate whether
$\Vol^+(X)=\Vol_{\text{BdFF}}(X,0)$.

In \cite{Ful}, Fulger introduces a different invariant
$\Vol_\text{F}(X,0)$ associated to an isolated singularity. It is
shown that $\Vol_{\text{BdFF}}(X,0)\geqslant \Vol_\text{F}(X,0)$
with equality if $X$ is $\mathbb{Q}$-Gorenstein. In \cite[Example
5.4]{BdFF}, an example is given where $\Vol_{\text{BdFF}}(X,0)>
\Vol_\text{F}(X,0)$.

It is conjectured that for a normal variety $X$ (may be not
$\mathbb{Q}$-Gorenstein) which has only isolated singularities,
there is a log canonical modification $f:Y\rightarrow X$ in the
sense that $K_Y+E_f$ is $f$-ample and $(Y,E_f)$ is log canonical. In
\cite[Proposition 2.4]{BH}, it is proved that if such modification
exists, then $\Vol_\text{BdFF}(X,0)=0$ if and only if $f$ is an
isomorphism in codimension 1.

\subsection*{Acknowledgement} The author would like to give special
thanks to Chenyang Xu for suggesting this problem and bringing the
papers \cite{OX} and \cite{Sho} to his attention. The author would
also like to thank Sebastien Boucksom, Tommaso de Fernex,
Christopher Hacon and Chenyang Xu for inspiring discussions.

\section{Preliminaries}

Throughout this paper, $X$ is a normal variety over $\mathbb{C}$.

\subsection{Valuations of $\mathbb{Q}$-divisors}

Let $X$ be a normal variety. A divisorial valuation $v$ on $X$ is a
discrete valuation of the function field of $X$ of the form
$v=q\val_F$ where $q$ is a positive integer and $F$ is a prime
divisor over $X$. Let $\mathcal{J}\subset\mathcal{K}$ be a finitely
generated sub-$\mathcal{O}_X$-module of the constant sheaf of
rational functions $\mathcal{K}=\mathcal{K}_X$ on $X$. For short, we
will refer to $\mathcal{J}$ as a fractional ideal sheaf on $X$.

The valuation $v(\mathcal{J})$ of a non-zero fractional ideal sheaf
$\mathcal{J}\subset\mathcal{K}$ along $v$ is given by
$$v(\mathcal{J})=\min\{v(\phi)|\phi\in\mathcal{J}(U),U\cap
c_X(v)\neq\emptyset\}.$$ The valuation $v(I)$ of a formal linear
combination $I=\sum a_k\cdot \mathcal{J}_k$ of fractional ideal
sheaves $\mathcal{J}_k\subset \mathcal{K}$ along $v$ is defined by
$v(I)=\sum a_k\cdot v(\mathcal{J}_k)$, where $a_k$ are real numbers.

The \textbf{$\natural$-valuation} (or \textbf{natural valuation})
along $v$ of a $\mathbb{R}$-Weil divisor $D$ on $X$ is
$v^\natural(D)=v(\mathcal{O}_X(-D))=v(\mathcal{O}_X(\lfloor
-D\rfloor))$. If $C$ is Cartier, then we have that
$v^\natural(C)=v(C)$ and $v^\natural(C+D)=v(C)+v^\natural(D)$. Note
also that, as $\mathcal{O}_X(D)\cdot \mathcal{O}_X(-D)\subseteq
\mathcal{O}_X$, we have that $v^\natural(D)+v^\natural(-D)\geqslant
0$.

To any non-trivial fractional ideal sheaf $\mathcal{J}$ on $X$, we
associate the divisor $\dvs(\mathcal{J})=\sum
\val_E^\natural(\mathcal{J})\cdot E$, where the sum is taken over
all prime divisors $E$ on $X$. Consider now a birational morphism
$f:Y\rightarrow X$ from a normal variety $Y$. For any divisor $D$ on
$X$, the \textbf{$\natural$-pullback} (or \textbf{natural pullback})
of $D$ to $Y$ is given by $f^\natural D=\dvs(\mathcal{O}_X(-D)\cdot
\mathcal{O}_Y)$. In the other words, $f^\natural D=\sum
\val_E^\natural(D)\cdot E$, where the sum is taken over all prime
divisors $E$ on $Y$. In particular, $\mathcal{O}_Y(-f^\natural
D)=(\mathcal{O}_X(-D)\cdot \mathcal{O}_Y)^{\vee\vee}$.

For every divisor $D$ on $X$ and every positive integer $m$, it is
shown in \cite[Lemma 2.8]{dFH} that $m\cdot v^\natural(D)\geqslant
v^\natural(mD)$ and
$$\inf_{k\geqslant 1}\frac{v^\natural(kD)}{k}=\liminf_{k\rightarrow
\infty}\frac{v^\natural(kD)}{k}=\lim_{k\rightarrow
\infty}\frac{v^\natural(k!D)}{k!}\in \mathbb{R}.$$ Let $D$ be a
$\mathbb{R}$-divisor on $X$. We define the \textbf{valuation along
$v$ of $D$} by $$v(D)=\lim_{k\rightarrow
\infty}\frac{v^\natural(k!D)}{k!}\in \mathbb{R}.$$

\begin{remark}
It is not hard to see that actually $$v(D)=\lim_{k\rightarrow
\infty}\frac{v^\natural(kD)}{k}\in \mathbb{R}.$$ See
\cite[Proposition 2.1]{BdFF}.
\end{remark}

\begin{remark}
Even if $D$ is a $\mathbb{Q}$-divisor on $X$, the valuation $v(D)$
may not be a rational number. See \cite[Section 3]{Urb}.
\end{remark}

If $f:Y\rightarrow X$ is a birational morphism from a normal variety
$Y$, then the \textbf{pullback of $D$} to $Y$ is defined by
$$f^*D=\sum\val_E(D)\cdot E,$$ where the sum is taken over all
prime divisors $E$ on $Y$. Notice that if $D$ is a
$\mathbb{Q}$-Cartier $\mathbb{Q}$-divisor and $m$ is a positive
integer such that $mD$ is Cartier, then $$v(D)=\frac{v(mD)}{m}
\quad\text{and}\quad f^*D=\frac{f^*(mD)}{m},$$ which coincide with
the usual valuation and pullback of $\mathbb{Q}$-Cartier
$\mathbb{Q}$-divisor. If $C$ is $\mathbb{Q}$-Cartier, then
$v(C+D)=v(C)+v(D)$ and $f^*(C+D)=f^*C+f^*D$.

\begin{lemma}\label{composition}
Let $f:Y\rightarrow X$ and $g:V\rightarrow Y$ be two birational
morphisms of normal varieties. Then, for every divisor $D$ on $X$,
the divisor $(fg)^\natural D-g^\natural(f^\natural D)$ is effective
and $g$-exceptional. Moreover, if $\mathcal{O}_X(-D)\cdot
\mathcal{O}_Y$ is an invertible sheaf, then $(fg)^\natural
D=g^\natural(f^\natural D)$. The similar statement applies to $f^*$
and $g^*$.
\end{lemma}

\begin{proof}
See Lemma 2.7 and Remark 2.13 in \cite{dFH}.
\end{proof}

\subsection{Relative canonical divisors}

We recall that a canonical divisor $K_X$ on a normal variety $X$ is,
by definition, the (componentwise) closure of any canonical divisor
of the regular locus of $X$. We also recall that $X$ is said to be
$\mathbb{Q}$-Gorenstein if some (equivalently, every) canonical
divisor $K_X$ is $\mathbb{Q}$-Cartier. For a proper birational
morphism $f:Y\rightarrow X$ of normal varieties, we fix a canonical
divisor $K_Y$ on $Y$ such that $f_*K_Y=K_X$. For any divisor $D$ on
$X$, we will write $D_Y$ for the strict transform $f_*^{-1}D$ of $D$
on $Y$.

For every $m\geqslant 1$, the \textbf{$m$-th limiting relative
canonical $\mathbb{Q}$-divisor} $K_{m,Y/X}$ of $Y$ over $X$ is
$$K_{m,Y/X}=K_Y-\frac{1}{m}f^\natural(mK_X).$$ The \textbf{relative
canonical $\mathbb{R}$-divisor} $K_{Y/X}$ of $Y$ over $X$ is
$$K_{Y/X}=K_Y-f^*K_X.$$ Clearly, $K_{Y/X}$ is the limsup of the
$\mathbb{Q}$-divisors $K_{m,Y/X}$. A $\mathbb{Q}$-divisor $\Delta$
on $X$ is said to be a \textbf{boundary}, if $\lfloor\Delta\rfloor =
0$ and $K_X+\Delta$ is $\mathbb{Q}$-Cartier. The \textbf{log
relative canonical $\mathbb{Q}$-divisor} of $(Y,\Delta_Y)$ over
$(X,\Delta)$ is given by
$$K_{Y/X}^{\Delta}=K_Y+\Delta_Y-f^*(K_X+\Delta).$$

\begin{remark}
Our definition of the relative canonical $\mathbb{R}$-divisor is
different from the one in \cite{dFH}. In this paper, the relative
canonical $\mathbb{R}$-divisor is defined as
$K_{Y/X}=K_Y+f^*(-K_X)$. And $K_Y-f^*K_X$ is denoted by $K^-_{Y/X}$.
It can be shown that, with this notation, $K_{Y/X}\geqslant
K^-_{Y/X}$. But they are not equal in general. See \cite[Example
3.4]{dFH}.
\end{remark}

For every integer $m\geqslant 1$, the \textbf{$m$-th limiting log
discrepancy $\mathbb{Q}$-divisor} $A_{m,Y/X}$ of $Y$ over $X$ is
$$A_{m,Y/X}=K_Y+E_f-\frac{1}{m}f^\natural(mK_X),$$ where $E_f$ is
the reduced exceptional divisor of $f$. The \textbf{log discrepancy
$\mathbb{R}$-divisor} $A_{Y/X}$ of $Y$ over $X$ is
$$A_{Y/X}=K_Y+E_f-f^*K_X.$$ The \textbf{log
discrepancy $\mathbb{Q}$-divisor} of $(Y,\Delta_Y)$ over
$(X,\Delta)$ is given by
$$A_{Y/X}^{\Delta}=K_Y+\Delta_Y+E_f-f^*(K_X+\Delta).$$

Consider a pair $(X,I=\sum a_k\cdot \mathcal{J}_k)$ where
$\mathcal{J}_k$ are non-zero fractional ideal sheaves on $X$ and
$a_k$ are real numbers. A \textbf{log resolution} of $(X,I)$ is a
proper birational morphism $f:Y\rightarrow X$ from a smooth variety
$Y$ such that for every $k$ the sheaf $\mathcal{J}_k\cdot
\mathcal{O}_Y$ is the invertible sheaf corresponding to a divisor
$E_k$ on $Y$, the exceptional locus $\Ex(f)$ of $f$ is also a
divisor, and $\Ex(f)\cup E$ has simple normal crossing, where
$E=\bigcup \Supp(E_k)$. If $\Delta$ is a boundary on $X$, then a log
resolution of the log pair $((X,\Delta),I)$ is given by a log
resolution $f:Y\rightarrow X$ of $(X,I)$ such that $\Ex(f)\cup E\cup
\Supp(f^*(K_X+\Delta))$ has simple normal crossings. The log
resolution always exists (see \cite[Theorem 4.2]{dFH}).

Let $X$ be a normal variety, and fix an integer $m\geqslant 2$.
Given a log resolution $f:Y\rightarrow X$ of
$(X,\mathcal{O}_X(-mK_X))$, a boundary $\Delta$ on $X$ is said to be
\textbf{$m$-compatible} for $X$ with respect to $f$ if:
\begin{enumerate}
\item $m\Delta$ is integral and $\lfloor\Delta\rfloor=0$,
\item $f$ is a log resolution for the log pair
$((X,\Delta);\mathcal{O}_X(-mK_X))$, and
\item $K_{Y/X}^{\Delta}=K_{m,Y/X}$.
\end{enumerate}

\begin{theorem}\label{compatible}
For any normal variety $X$, any integer $m\geqslant 2$ and any log
resolution $f:Y\rightarrow X$ of $(X,\mathcal{O}_X(-mK_X))$, there
exists an $m$-compatible boundary $\Delta$ for $X$ with respect to
$f$.
\end{theorem}
\begin{proof}
See \cite[Theorem 5.4]{dFH}.
\end{proof}

\subsection{Shokurov's $b$-divisors}

Let $X$ be a normal variety. The set of all proper birational
morphisms $\pi:X_\pi\rightarrow X$ from a normal variety $X_\pi$
modulo isomorphism is (partially) ordered by $\pi'\geqslant \pi$ if
and only if $\pi'$ factors through $\pi$, and any two proper
birational morphisms can be dominated by a third one. The
\textbf{Riemann-Zariski space} $X$ is defined as the projective
limit, $\mathcal{X}=\varprojlim_\pi X_\pi$. The group of
\textbf{Weil $b$-divisors} over $X$ is defined as
$\Div(\mathcal{X})=\varprojlim_\pi \Div(X_\pi)$, where $\Div(X_\pi)$
denotes the group of Weil divisors on $X_\pi$ and the limit is taken
with respect to the push-forwards. The group of \textbf{Cartier
$b$-divisors} over $X$ is defined as
$\CDiv(\mathcal{X})=\varinjlim_\pi \CDiv(X_\pi)$, where
$\CDiv(X_\pi)$ denotes the group of Cartier divisors on $X_\pi$ and
the limit is taken with respect to the pull-backs. An element in
$\Div_\mathbb{R}(\mathcal{X})=\Div(\mathcal{X})\otimes \mathbb{R}$
(resp. $\CDiv_\mathbb{R}(\mathcal{X})=\CDiv(\mathcal{X})\otimes
\mathbb{R}$) will be called an $\mathbb{R}$-Weil $b$-divisor (resp.
$\mathbb{R}$-Cartier $b$-divisor), and similarly with $\mathbb{Q}$
in place of $\mathbb{R}$. Clearly, a Weil $b$-divisor $W$ over $X$
consists of a family of Weil divisors $W_\pi\in \Div(X_\pi)$ that
are compatible under push-forward. We say that $W_\pi$ is the
\textbf{trace} of $W$ on the model $X_\pi$. Let $C$ be an Cartier
$b$-divisor. We say that $\pi:X_\pi\rightarrow X$ is a
\textbf{determination} of $C$, if $C$ can be obtained by pulling
back $C_\pi$ to models dominating $\pi$ and pushing forward to other
models, in which case we denote $C=\overline{C_\pi}$.

Let $Z$ and $W$ be two $\mathbb{R}$-Weil $b$-divisors over $X$. We
say that $Z\leqslant W$, if for any model $\pi:X_\pi\rightarrow X$
we have $Z_\pi\leqslant W_\pi$. We say an $\mathbb{R}$-Cartier
$b$-divisor is \textbf{relatively nef} over $X$, if its trace is
relatively nef on one (hence any sufficiently high) determination.
An $\mathbb{R}$-Weil $b$-divisor $W$ is \textbf{relatively nef} over
$X$ if and only if there is a sequence of relatively nef
$\mathbb{R}$-Cartier $b$-divisors over $X$ such that the traces
converge to the trace of $W$ in the numerical class over $X$ on each
model. We will need the following variant of the Negativity Lemma in
the future.

\begin{lemma}\label{negativity}
Let $W$ be an relatively nef $\mathbb{R}$-Weil $b$-divisor over $X$.
Let $\pi:X_\pi\rightarrow X$ and $\pi':X_{\pi'}\rightarrow X$ be two
models over $X$ such that $\pi'$ factor through $\pi$ via
$\rho:X_{\pi'}\rightarrow X_\pi$. Then $W_{\pi'}\leqslant
-\rho^*(-W_\pi)$.
\end{lemma}

\begin{proof}
See \cite[Proposition 2.12]{BdFF}.
\end{proof}

Let $C_1,\ldots,C_n$ be $\mathbb{R}$-Cartier $b$-divisors, where
$n=\dim X$. Let $f$ be a common determination. It is clear that the
intersection number $C_{1,f}\cdot\ldots\cdot C_{n,f}$ is independent
of the choice of $f$ by the projection formula. We define
$C_1\cdot\ldots\cdot C_n$ to be the above intersection number. If
$W_1,\ldots,W_n$ are relatively nef $\mathbb{R}$-Weil $b$-divisors
over $X$, we define
$$W_1\cdot\ldots\cdot W_n=\inf (C_1\cdot\ldots\cdot C_n)\in
[-\infty,\infty),$$ where the infimum is taken over all relatively
nef $\mathbb{R}$-Cartier $b$-divisors $C_i$ over $X$ such that
$C_i\geqslant W_i$ for each $i$. It is obvious that the intersection
number is monotonic in the sense that, if $W_i\leqslant W_i'$ for
each $i$, then $W_1\cdot\ldots\cdot W_n\leqslant
W_1'\cdot\ldots\cdot W_n'$. For further properties of the
intersection number, we refer to Section 4.3 and Appendix A of
\cite{BdFF}.

Given a canonical divisor $K_X$ on $X$, there is a unique canonical
divisor $K_{X_\pi}$ for each model $\pi:X_\pi\rightarrow X$ with the
property that $\pi_*K_{X_{\pi}}=K_X$. Hence, a choice of $K_X$
determines a \textbf{canonical $b$-divisor} $K_\mathcal{X}$ over
$X$.

\subsection{Nef envelopes}

The \textbf{nef envelope $\Env_X(D)$ of an $\mathbb{R}$-divisor} $D$
on $X$ is an $\mathbb{R}$-Weil $b$-divisor over $X$ whose trace on a
model $\pi:X_\pi\rightarrow X$ is $-\pi^*(-D)$. For more details,
see \cite[Section 2]{BdFF}. If $D$ is $\mathbb{Q}$-Cartier, then
$\Env_X(D)$ is the $\mathbb{Q}$-Cartier $b$-divisor $\overline{D}$.

The \textbf{nef envelope $\Env_\mathcal{X}(W)$ of an
$\mathbb{R}$-Weil $b$-divisor} $W$ over $X$ is the largest
relatively nef $\mathbb{R}$-Weil $b$-divisor $Z$ over $X$ such that
$Z\leqslant W$. It is well-defined by \cite[Proposition 2.15]{BdFF}.
It is clear that if $W_1\leqslant W_2$, then
$\Env_\mathcal{X}(W_1)\leqslant\Env_\mathcal{X}(W_2)$.

The \textbf{log discrepancy $b$-divisor} is defined as
$$A_{\mathcal{X}/X}=K_{\mathcal{X}}+E_{\mathcal{X}/X}+\Env_X(-K_X),$$
where the trace of $E_{\mathcal{X}/X}$ in any model $\pi$ is equal
to the reduced exceptional divisor $E_\pi$ over $X$. It is clear
that the trace of $A_{\mathcal{X}/X}$ on a model
$\pi:X_\pi\rightarrow X$ is $A_{X_\pi/X}$. Similarly, for every
integer $m\geqslant 1$, we define the \textbf{$m$-th limiting log
discrepancy $b$-divisor} $A_{m,\mathcal{X}/X}$ to be a
$\mathbb{Q}$-Weil $b$-divisor whose trace on a model
$\pi:X_\pi\rightarrow X$ is $A_{m,X_\pi/X}$. It is easy to check
that $A_{m,\mathcal{X}/X}\leqslant A_{\mathcal{X}/X}$ and
$A_{\mathcal{X}/X}$ is the limsup of $A_{m,\mathcal{X}/X}$.

The \textbf{volume of the singularity} on $X$ is defined by
$$\Vol_{\text{BdFF}}(X,0)=-\Env_\mathcal{X}(A_{\mathcal{X}/X})^n.$$
It is shown in \cite{BdFF} that if $X$ has isolated singularity,
then $\Vol_{\text{BdFF}}(X,0)$ is a well-defined non-negative finite
real number.

\subsection{Log canonical modification}

Suppose $(X,\Delta)$ is a pair such that $X$ is a normal variety,
$\Delta$ is an effective $\mathbb{Q}$-divisor and $K_X+\Delta$ is
$\mathbb{Q}$-Cartier. A birational projective morphism
$f:Y\rightarrow X$ is called a log canonical modification of
$(X,\Delta)$ if
\begin{enumerate}
\item $(Y,\Delta_Y+E_f)$ is log canonical,
\item $K_Y+\Delta_Y+E_f$ is $f$-ample,
\end{enumerate}
where $\Delta_Y$ is the strict transform of $\Delta$ and $E_f$ is
the reduced exceptional divisor of $f$. It is shown in \cite{OX}
that the log canonical modification exists uniquely up to
isomorphism for any log pair $(X,\Delta)$. Clearly, if
$f':Y'\rightarrow X$ is a log resolution of the pair $(X,\Delta)$,
then
$$Y\cong\BigProj_X\bigoplus_{m\in\mathbb{Z}_{\geqslant
0}}f_*\mathcal{O}_{Y'}(m(K_{Y'}+\Delta_{Y'}+E_{f'})).$$

\begin{lemma}\label{exc}
Let $(X,\Delta)$ be a pair as above which is not log canonical. Let
$f:Y\rightarrow X$ be the log canonical model. Write
$f^*(K_X+\Delta)\sim_\mathbb{Q} K_Y+\Delta_Y+B$, and $B =\sum
b_iB_i$ as the sum of distinct prime divisors such that
$f_*(B)=\Delta$. We let $B^{>1}$ be the nonzero divisor
$\sum_{b_i>1} b_iB_i$, then $\Supp(B^{>1}) = \Ex(f)$. In particular,
$\Ex(f)\subset Y$ is of pure codimension 1.
\end{lemma}

\begin{proof}
See \cite[Lemma 2.4]{OX}.
\end{proof}

\section{$\mathbb{Q}$-Gorenstein Case}

Assume that $X$ is a $\mathbb{Q}$-Gorenstein normal variety with
isolated singularities. We pick $\Delta=0$ and suppose that
$f:Y\rightarrow X$ is the log canonical modification of $X$. Let
$F=f^*K_X-K_Y-E_f$. We define $\Vol(X)=-(K_Y+E_f-f^*K_X)^n$. By the
Negativity Lemma (see \cite[Lemma 3.39]{KM}), $F\geqslant 0$. Since
$K_Y+E_f$ is $f$-ample and $F$ is $f$-exceptional, we have that
$$\Vol(X)=-(K_Y+E_f-f^*K_X)^n=F\cdot (K_Y+E_f)^{n-1}\geqslant
0.$$

\begin{remark}
This definition can be extended to the case that $X$ has isolated
non-log-canonical locus.
\end{remark}

\begin{theorem}\label{QGor}
If $X$ is a $\mathbb{Q}$-Gorenstein normal variety which has
isolated singularities, then $\Vol_{\text{BdFF}}(X,0)=\Vol(X)$.
\end{theorem}

\begin{proof}
We show that $\Env_\mathcal{X}(A_{\mathcal{X}/X})$ is a
$\mathbb{Q}$-Cartier $b$-divisor and equals to $\overline{A_{Y/X}}$
where $f:Y\rightarrow X$ is the log canonical modification of
$(X,0)$. Then the theorem follows immediately.

We only need to show that on a high enough model $f':Y'\rightarrow
X$ which factors through $f:Y\rightarrow X$ via $g:Y'\rightarrow Y$,
we have that $D=\Env_\mathcal{X}(A_{\mathcal{X}/X})_{f'}$ equals to
$g^*A_{Y/X}$.

First, we show that $g^*A_{Y/X}\leqslant D$. Since $(Y,E_f)$ is log
canonical, we have that
$$\begin{array}{cl}
  & A_{Y'/X}-g^*A_{Y/X} \\
=\!\!\! & K_{Y'}+E_{f'}-f'^*K_X-g^*(K_Y+E_f-f^*K_X) \\
=\!\!\! & K_{Y'}+(E_f)_{Y'}+E_g-g^*(K_Y+E_f) \\
\geqslant\!\!\! & 0,
\end{array}$$ where $(E_f)_{Y'}$ is the strict transform of $E_f$ on $Y'$. As
$A_{Y/X}$ is $f$-ample, we have that $g^*A_{Y/X}$ is $f'$-nef. We
can conclude that $\overline{A_{Y/X}}$ is a relatively nef
$\mathbb{Q}$-Cartier $b$-divisor such that
$\overline{A_{Y/X}}\leqslant A_{\mathcal{X}/X}$. By the definition
of nef envelope, we have that
$\overline{A_{Y/X}}\leqslant\Env_\mathcal{X}(A_{\mathcal{X}/X})$. In
particular, $g^*A_{Y/X}\leqslant D$.

On the other hand, by the definition of nef envelope, we have that
$\Env_\mathcal{X}(A_{\mathcal{X}/X})$ is relatively nef over $X$. We
may apply Lemma \ref{negativity}. Thus, $D\leqslant
-g^*(-\Env_\mathcal{X}(A_{\mathcal{X}/X})_f)$. By definition,
$\Env_\mathcal{X}(A_{\mathcal{X}/X})_f\leqslant A_{Y/X}$. Hence,
$D\leqslant -g^*(-A_{Y/X})=g^*A_{Y/X}$, since $A_{Y/X}$ is
$\mathbb{Q}$-Cartier. Therefore, $D=g^*A_{Y/X}$.
\end{proof}

\begin{theorem}\label{DCC}
There exists a positive lower bound, only depending on the
dimension, for the volume of isolated Gorenstein singularities with
positive volume.
\end{theorem}

\begin{proof}
Suppose $X$ is a Gorenstein normal variety with isolated
singularities and $f:Y\rightarrow X$ is its log canonical
modification. Let $F=f^*K_X-K_Y-E_f=\sum a_i\cdot E_i$, where $E_i$
are $f$-exceptional divisors. Since $X$ is Gorenstein, we have that
$a_i$ are positive integers by Lemma \ref{exc}. If $\Vol(X)>0$, then
$F\neq0$, hence $F\geqslant E_f$. We have that
$$\Vol(X)\geqslant E_f\cdot
(K_Y+E_f)^{n-1}=((K_Y+E_f)|_{E_f})^{n-1}=(K_{E_f}+\text{Diff}_{E_f}(0))^{n-1}.$$
Since $(Y,E_f)$ is log canonical, by \cite[16.6]{Kol}, the
coefficients of $\text{Diff}_{E_f}(0)$ lies in
$\{0,1\}\cup\{1-\frac{1}{m}|m\geqslant 2\}$, which is a DCC set. By
\cite[Theorem 1.3]{HMX12}, we have that
$(K_{E_f}+\text{Diff}_{E_f}(0))^{n-1}$ lies in a DCC set. The
theorem follows.
\end{proof}

\section{Non-$\mathbb{Q}$-Gorenstein Case}

Let $X$ be a normal variety which has only isolated singularities.
For any integer $m\geqslant 2$, fix a log resolution $f:Y\rightarrow
X$ of $(X,\mathcal{O}_X(-mK_X))$. By Theorem \ref{compatible}, we
can find a boundary $\Delta$ such that $K_{Y/X}^{\Delta}=K_{m,Y/X}$.
Let $f_{lc}:Y_{lc}\rightarrow X$ be the log canonical modification
of the pair $(X,\Delta)$. Then
$$Y_{lc}\cong\BigProj_X\bigoplus_{m\in\mathbb{Z}_{\geqslant
0}}f_*\mathcal{O}_Y(m(K_Y+\Delta_Y+E_f)).$$ Assuming that $\Delta'$
is another $m$-compatible boundary for $X$ with respect to $f$ and
$f'_{lc}:Y'_{lc}\rightarrow X$ is the corresponding log canonical
modification, we have that $K_{Y/X}^{\Delta}=K_{Y/X}^{\Delta'}$.
Hence, $\Delta_Y-\Delta'_Y=f^*(\Delta-\Delta')$. Now,
$$\begin{array}{cl}
  & f_*\mathcal{O}_Y(m(K_Y+\Delta'_Y)) \\
= & f_*\mathcal{O}_Y(m(K_Y+\Delta_Y-f^*(\Delta-\Delta'))) \\
= & f_*\mathcal{O}_Y(m(K_Y+\Delta_Y))\otimes
\mathcal{O}_X(m(\Delta'-\Delta))
\end{array}$$ for sufficiently divisible $m$, as $\Delta-\Delta'$ is $\mathbb{Q}$-Cartier.
Thus, there is a natural $X$-isomorphism $\sigma:Y_{lc}\rightarrow
Y'_{lc}$ such that $f_{lc}=f'_{lc}\circ \sigma$. Fix a common
resolution of $Y$ and $Y_{lc}$, $\tilde{f}:\tilde{Y}\rightarrow X$,
as in the following diagram:
$$\xymatrix{
& \tilde{Y} \ar[ld]_{s} \ar[rd]^{t} & \\
Y \ar@{-->}[rr]\ar[rd]_{f} & & Y_{lc} \ar[ld]^{f_{lc}} \\
& X & }$$ Noticing that $\tilde{Y}$ is also a common resolution of
$Y$ and $Y'_{lc}$, we have that the morphism $s:\tilde{Y}\rightarrow
Y$ is independent of the choice of $\Delta$.

\begin{theorem}\label{NQGor}
The $\mathbb{R}$-Weil $b$-divisor
$\Env_\mathcal{X}(A_{m,\mathcal{X}/X})$ is a $\mathbb{Q}$-Cartier
$b$-divisor. If $\Delta$ is $m$-compatible for $X$ with respect to
$\tilde{f}:\tilde{Y}\rightarrow X$, then
$$\Env_\mathcal{X}(A_{m,\mathcal{X}/X})=\overline{A_{Y_{lc}/X}^\Delta}.$$
\end{theorem}
\begin{proof}
We will mimic the proof of Theorem \ref{QGor}. We only need to show
that on a high enough model $\rho:Z\rightarrow X$ which factors
through $\tilde{f}:\tilde{Y}\rightarrow X$ by $\pi:Z\rightarrow
\tilde{Y}$, we have that
$D=\Env_\mathcal{X}(A_{m,\mathcal{X}/X})_{\rho}$ equals to $(t\circ
\pi)^*A_{Y_{lc}/X}^\Delta$.

First, we show that $(t\circ \pi)^*A_{Y_{lc}/X}^\Delta\leqslant D$.
Since $\Delta$ is $m$-compatible for $X$ with respect to
$\tilde{f}$, we have that
$$K_{\tilde{Y}}+\Delta_{\tilde{Y}}-\tilde{f}^*(K_X+\Delta)=K_{\tilde{Y}}-\frac{1}{m}\tilde{f}^\natural(mK_X).$$
Hence,
$$-\frac{1}{m}\rho^\natural(mK_X)=\pi^*\left(-\frac{1}{m}\tilde{f}^\natural(mK_X)\right)=\pi^*\Delta_{\tilde{Y}}-\rho^*(K_X+\Delta),$$
by Lemma \ref{composition} and the fact that $\tilde{Y}$ is smooth.
Now, the difference
$$\begin{array}{cl}
\vspace{0.2cm}        & A_{m,Z/X}-(t\circ\pi)^*A_{Y_{lc}/X}^\Delta \\
\vspace{0.2cm}=\!\!\! & K_Z+E_\rho-\displaystyle\frac{1}{m}\rho^\natural(mK_X)-(t\circ\pi)^*(K_{Y_{lc}}+\Delta_{Y_{lc}}+E_{f_{lc}})+\rho^*(K_X+\Delta) \\
\vspace{0.2cm}=\!\!\! &
K_Z+E_\rho+\pi^*\Delta_{\tilde{Y}}-(t\circ\pi)^*(K_{Y_{lc}}+\Delta_{Y_{lc}}+E_{f_{lc}})
\\
\vspace{0.2cm}=\!\!\! & (K_Z+\Delta_Z+(E_{f_{lc}})_Z+E_{t\circ
\pi}-(t\circ\pi)^*(K_{Y_{lc}}+\Delta_{Y_{lc}}+E_{f_{lc}}))+(\pi^*\Delta_{\tilde{Y}}-\Delta_Z)
\end{array}$$
Since $(Y_{lc},\Delta_{Y_{lc}}+E_{f_{lc}})$ is log canonical, we
have that the first term is effective and $(t\circ\pi)$-exceptional.
As $\Delta_Z$ is the strict transform of the effective divisor
$\Delta_{\tilde{Y}}$, the second term is effective and
$\pi$-exceptional, hence, $(t\circ\pi)$-exceptional. We conclude
that $A_{m,Z/X}-(t\circ\pi)^*A_{Y_{lc}/X}^\Delta$ is effective and
$(t\circ\pi)$-exceptional. Since $A_{Y_{lc}/X}^\Delta$ is $f$-ample,
we have that $(t\circ\pi)^*A_{Y_{lc}/X}^\Delta$ is $\rho$-nef. We
can conclude that $\overline{A_{Y_{lc}/X}^\Delta}$ is a relatively
nef $\mathbb{Q}$-Cartier $b$-divisor such that
$\overline{A_{Y_{lc}/X}^\Delta}\leqslant A_{m,\mathcal{X}/X}$. By
the definition of nef envelope, we have that
$\overline{A_{Y_{lc}/X}^\Delta}\leqslant\Env_\mathcal{X}(A_{m,\mathcal{X}/X})$.
In particular, $(t\circ \pi)^*A_{Y_{lc}/X}^\Delta\leqslant D$.

On the other hand, by the definition of nef envelope, we have that
$\Env_\mathcal{X}(A_{m,\mathcal{X}/X})$ is relatively nef over $X$.
We may apply Lemma \ref{negativity}. Thus, $D\leqslant -(t\circ
\pi)^*(-\Env_\mathcal{X}(A_{m,\mathcal{X}/X})_{f_{lc}})$. By
definition, $\Env_\mathcal{X}(A_{m,\mathcal{X}/X})_{f_{lc}}\leqslant
A_{m,Y_{lc}/X}$. Hence, $D\leqslant
-(t\circ\pi)^*(-A_{m,Y_{lc}/X})=(t\circ \pi)^*A_{Y_{lc}/X}^\Delta$,
since $A_{Y_{lc}/X}^\Delta$ is $\mathbb{Q}$-Cartier. Therefore,
$D=(t\circ\pi)^*A_{Y_{lc}/X}^\Delta$.
\end{proof}

We can define the volume of singularities of $X$ as follow:

\begin{definition}
The $m$-th limiting volume of singularity of $X$ is
$$\Vol_m(X)=-\Env_\mathcal{X}(A_{m,\mathcal{X}/X})^n.$$
\end{definition}

\begin{corollary}\label{vol_m}
In the setting of Theorem \ref{NQGor}, if $\Delta$ is $m$-compatible
for $X$ with respect to $\tilde{f}$, then
$$\Vol_m(X)=-(A_{Y_{lc}/X}^{\Delta})^n=-A_{Y_{lc}/X}^{\Delta}\cdot
(K_{Y_{lc}}+\Delta_{Y_{lc}}+E_{f_{lc}})^{n-1}\geqslant 0.$$
\end{corollary}

\begin{proof}
The first equation is straightforward by Theorem \ref{NQGor} and the
definition of intersection number. The second equation is valid
since $A_{Y_{lc}/X}^{\Delta}$ is $f_{lc}$-exceptional. By the
Negativity Lemma, we have that $A_{Y_{lc}/X}^{\Delta}\leqslant 0$.
Since $K_{Y_{lc}}+\Delta_{Y_{lc}}+E_{f_{lc}}$ is $f_{lc}$-ample, we
have the inequality in the corollary.
\end{proof}

For an arbitrary boundary $\Delta$ on $X$, we have the following
inequalities.

\begin{proposition}\label{arbi_delta}
Suppose that $\Delta$ is a boundary on $X$, $m$ is the index of
$K_X+\Delta$ and $f:Y\rightarrow X$ is the log canonical
modification of $(X,\Delta)$. Then
\begin{enumerate}
\item $\Env_\mathcal{X}(A_{m,\mathcal{X}/X})\geqslant
\overline{A_{Y/X}^{\Delta}}$,
\item $\Vol_m(X)\leqslant -(A_{Y/X}^{\Delta})^n$.
\end{enumerate}
\end{proposition}

\begin{proof}
For any model $\pi:X_\pi\rightarrow X$, we have that
$$\pi^*(m(K_X+\Delta))+\pi^\natural(-m\Delta)=\pi^\natural(mK_X).$$ Hence,
$$\begin{array}{lcl}
\vspace{0.2cm} A_{m,X_\pi/X} & \!\!\!=\!\!\! & K_{X_\pi}+E_\pi-\displaystyle\frac{1}{m}\pi^\natural(mK_X) \\
\vspace{0.2cm}           & \!\!\!=\!\!\! &
          K_{X_\pi}+E_\pi-\displaystyle\frac{1}{m}\pi^\natural(-m\Delta)-\frac{1}{m}\pi^*(m(K_X+\Delta))
          \\
\vspace{0.2cm}           & \!\!\!\geqslant\!\!\! &
K_{X_\pi}+E_\pi+\Delta_{X_\pi}-\pi^*(K_X+\Delta) =
A_{X_\pi/X}^\Delta.
\end{array}$$
Thus, as $b$-divisors, $A_{m,\mathcal{X}/X}\geqslant
A_{\mathcal{X}/X}^\Delta$, hence,
$$\Env_\mathcal{X}(A_{m,\mathcal{X}/X})\geqslant \Env_\mathcal{X}(A_{\mathcal{X}/X}^\Delta).$$

On the other hand, for any model $f':Y'\rightarrow Y$ factoring
through $f$ via $g:Y'\rightarrow Y$, we have that
$A_{Y'/X}^\Delta\geqslant g^*A_{Y/X}^\Delta$, since $(Y,
\Delta_Y+E_f)$ is log canonical. As $b$-divisors,
$A_{\mathcal{X}/X}^\Delta\geqslant \overline{A_{Y/X}^\Delta}$. As
$K_Y+\Delta_Y+E_f$ is $f$-ample, we have that
$\overline{A_{Y/X}^\Delta}$ is relatively nef over $X$. Thus,
$\Env_\mathcal{X}(A_{\mathcal{X}/X}^\Delta)\geqslant
\overline{A_{Y/X}^\Delta}$. We proved (1).

Since both $\Env_\mathcal{X}(A_{m,\mathcal{X}/X})$ and
$\overline{A_{Y/X}^\Delta}$ are relatively nef and exceptional over
$X$, (2) follows from the inequality between intersection numbers.
\end{proof}

\begin{remark}
In the last proposition, one can show that
$\Env_\mathcal{X}(A_{\mathcal{X}/X}^\Delta)=\overline{A_{Y/X}^\Delta}$.
\end{remark}

For any two positive integers $m$ and $l$ and any model
$f:Y\rightarrow X$, since
$$\frac{1}{m}f^\natural(mK_X)\geqslant
\frac{1}{lm}f^\natural(lmK_X)\geqslant f^*K_X,$$ we have that
$$A_{m,Y/X}\leqslant A_{lm,Y/X}\leqslant A_{Y/X},$$ hence
$$A_{m,\mathcal{X}/X}\leqslant A_{lm,\mathcal{X}/X}\leqslant A_{\mathcal{X}/X}.$$
By the definition of nef envelope, we have that
$$\Env_\mathcal{X}(A_{m,\mathcal{X}/X})\leqslant
\Env_\mathcal{X}(A_{lm,\mathcal{X}/X})\leqslant
\Env_\mathcal{X}(A_{\mathcal{X}/X}).$$ Since they are both
exceptional over $X$ by Theorem \ref{NQGor}, we have the following
inequality of volumes:
$$\Vol_m(X)\geqslant \Vol_{lm}(X)\geqslant \Vol_{\text{BdFF}}(X,0).$$

\begin{corollary}\label{vol=lc}
The following statements are equivalent:
\begin{enumerate}
\item There exists a boundary $\Delta$ on $X$ such that $(X,\Delta)$
is log canonical.
\item $\Vol_m(X)=0$ for some (hence any multiple of) integer
$m\geqslant 1$.
\end{enumerate}
\end{corollary}

\begin{proof}
$(1)\Rightarrow (2)$. Suppose $m(K_X+\Delta)$ is Cartier. By
Proposition \ref{arbi_delta}, $$A_{m,\mathcal{X}/X}\geqslant
A_{\mathcal{X}/X}^\Delta\geqslant 0,$$ since $(X,\Delta)$ is log
canonical. As $0$ is a relatively nef $b$-divisor over $X$, we have
that $\Env_\mathcal{X}(A_{m,\mathcal{X}/X})\geqslant 0$. On the
other hand, by Theorem \ref{NQGor} and the Negativity Lemma,
$\Env_\mathcal{X}(A_{m,\mathcal{X}/X})\leqslant 0$. Hence,
$\Env_\mathcal{X}(A_{m,\mathcal{X}/X})=0$. We can conclude that
$\Vol_m(X)=0$.

$(2)\Rightarrow (1)$. Let $\Delta$ be an $m$-compatible boundary for
$X$ with respect to $\tilde{f}$ in the setting of Theorem
\ref{NQGor}. By Theorem \ref{compatible}, such a boundary always
exists. Since $\Vol_m(X)=0$, by Corollary \ref{vol_m}, we have that
$$-A_{Y_{lc}/X}^{\Delta}\cdot
(K_{Y_{lc}}+\Delta_{Y_{lc}}+E_{f_{lc}})^{n-1}=0.$$ Since
$K_{Y_{lc}}+\Delta_{Y_{lc}}+E_{f_{lc}}$ is $f_{lc}$-ample, this is
equivalent to $A_{Y_{lc}/X}^{\Delta}=0$. Thus, we have that
$$f_{lc}^*(K_X+\Delta)=K_{Y_{lc}}+\Delta_{Y_{lc}}+E_{f_{lc}}.$$ For
any model $\rho:Z\rightarrow X$ factoring through $f_{lc}$ via
$\pi:Z\rightarrow Y_{lc}$, we have that
$$\begin{array}{rcl}
A_{Z/X}^\Delta & \!\!\!=\!\!\! &
                 K_Z+\Delta_Z+E_\rho-\rho^*(K_X+\Delta) \\
               & \!\!\!=\!\!\! &
                 K_Z+\Delta_Z+(E_{f_{lc}})_Z+E_\pi-\pi^*(K_{Y_{lc}}+\Delta_{Y_{lc}}+E_{f_{lc}})
                 \\
               & \!\!\!\geqslant\!\!\! & 0,
\end{array}$$
since $(Y_{lc},\Delta_{Y_{lc}}+E_{f_{lc}})$ is log canonical.
Therefore, $(X,\Delta)$ is log canonical.
\end{proof}

\begin{definition}
The augmented volume of singularities on $X$ is
$$\Vol^+(X)=\liminf_m \Vol_m(X)=\lim_{k\rightarrow
\infty}\Vol_{k!}(X)\geqslant \Vol_{\text{BdFF}}(X,0).$$
\end{definition}

\begin{remark}
While it is proved in the appendix of \cite{BdFF} that the
intersection number is continuous, it is not clear that
$\Env_\mathcal{X}(A_{m,\mathcal{X}/X})$ converge to
$\Env_\mathcal{X}(A_{\mathcal{X}/X})$. It is interesting to have an
example with $\Vol^+(X)>\Vol_{\text{BdFF}}(X,0)$.
\end{remark}

\subsection{Cone singularities}\label{cone_sing} We will give a counterexample to
Problem B in this section.

Let $(V,H)$ be a non-singular projective polarized variety of
dimension $n-1$. The vertex 0 is the isolated singularity of the
variety
$$X=\Spec\bigoplus_{m\geqslant 0}H^0(V,\mathcal{O}_V(mH)).$$ We
assume that $H$ is sufficiently ample so that $X$ is normal. Blowing
up 0 gives a resolution of singularities for $X$ that we denote by
$Y$. The induced map $f:Y\rightarrow X$ is isomorphic to the
contraction of the zero section $E$ of the total space of the vector
bundle $\mathcal{O}_V(H)$. Let $\pi:Y\rightarrow V$ be the bundle
map. We have that $E\cong V$. The co-normal bundle of $E$ in $Y$ is
$$\mathcal{O}_E(-E)\cong \mathcal{O}_V(H).$$

Let us slightly change our notation from previous sections. Let
$\Gamma$ be a boundary on $X$, $\Gamma_Y$ be the strict transform of
$\Gamma$ on $Y$ and $\Delta=\Gamma_Y|_E$. Since $A^\Gamma_{Y/X}$ is
exceptional, we may assume that $A^\Gamma_{Y/X}=-aE$ for some
rational number $a$. Restricting to $E$, we have that
$$K_V+\Delta\sim_\mathbb{Q} aH,$$ by the adjunction formula.
On the other hand, assuming that $\Delta$ is an effective
$\mathbb{Q}$-Cartier divisor on $V$ such that
$\Delta\sim_\mathbb{Q}-K_V+aH$, we may set $\Gamma=C_\Delta$ and get
that $A_{Y/X}^{C_\Delta}=-aE$, where $C_\Delta$ is the cone over
$\Delta$ in $X$.

Let $C$ be an elliptic curve, $U$ be a semi-stable vector bundle on
$C$ of rank 2 and degree 0 and $V=\mathbb{P}(U)$ be the ruled
surface over $C$. The nef cone $\text{Nef}(V)$ and pseudo-effective
cone $\overline{\text{NE}}(V)$ are the same. They are spanned by the
section $C_0$ corresponding to the tautalogical bundle
$\mathcal{O}_{\mathbb{P}(U)}(1)$ and a fiber $F$ of the ruling (for
details, see e.g. \cite[Section 1.5.A]{Laz}). Moreover, as in
\cite[Example 1.1]{Sho}, if $C'$ is an effective curve on $V$ such
that $C'\equiv mC_0$ for some positive integer $m$, then $C'=mC_0$.

\begin{theorem}\label{counterexample}
Let $V$ be the ruled surface as above. Fix an ample divisor $H$ on
$V$. Let $X$ be the affine cone over $(V,H)$. Suppose $H$ is
sufficiently ample so that $X$ is normal. Then $\Vol^+(X)=0$, hence
$\Vol_\text{BdFF}(X,0)=0$. But there is no effective
$\mathbb{Q}$-divisor $\Gamma$ such that $(X,\Gamma)$ is log
canonical.
\end{theorem}

\begin{proof}
Fix an ample divisor $H$ on $V$. Since $K_V\sim -2C_0$, we have that
$-K_V+aH$ is ample for any rational number $a>0$. Let $D$ be a
smooth curve in $|n(-K_V+aH)|$ for some sufficiently large positive
integer $n$, and set $\Delta=\frac{1}{n}D$. Then $(Y,\pi^*\Delta+E)$
is the log canonical modification of $(X,C_\Delta)$, since
$(K_Y+\pi^*\Delta+E)|_E\sim_\mathbb{Q}aH$ is ample and
$(X,C_\Delta)$ is not log canonical. Suppose $m$ is the index of
$K_X+C_\Delta$. By Proposition \ref{arbi_delta} (2),
$$\Vol_m(X)\leqslant -(A_{Y/X}^{C_\Delta})^3 = (aE)^3 = a^3H^2.$$
As $a\rightarrow 0$, we conclude that $\Vol^+(X)=0$, hence
$\Vol_\text{BdFF}(X,0)=0$.

If $(X,\Gamma)$ is log canonical for some effective
$\mathbb{Q}$-divisor $\Gamma$ on $X$, then
$A_{Y/X}^{\Gamma}=-aE\geqslant 0$, hence $a\leqslant 0$. On the
other hand, let $\Delta=\Gamma_Y|_E\in\overline{\text{NE}}(V)$. Then
$$\Delta \sim_\mathbb{Q} -K_V+aH\sim_\mathbb{Q} 2C_0+aH.$$ Since
$\Delta\geqslant 0$, we have $a\geqslant 0$ and hence $a=0$. Thus
$A_{Y/X}^{\Gamma}=0$, and $(Y,\Gamma_Y+E)$ is log canonical. But
$\Delta=\Gamma_Y|_E$ is an effective $\mathbb{Q}$-divisor linearly
equivalent to $2C_0$, and hence $\Delta=2C_0$, a contradiction.
\end{proof}

In \cite[Definition 7.1]{dFH}, a normal variety $X$ is defined to be
log canonical if for one (hence any sufficiently divisible) positive
integer $m$, the $m$-th limiting log discrepancy $b$-divisor
$A_{m,\mathcal{X}/X}\geqslant 0$. And in [ibid, Proposition 7.2],
they proved that $X$ is log canonical if and only if there is a
boundary $\Delta$ such that the pair $(X,\Delta)$ is log canonical.
It is natural to ask whether this definition is equivalent to the
one requiring that the log discrepancy $b$-divisor
$A_{\mathcal{X}/X}\geqslant 0$.

\begin{corollary}
Let $X$ be the affine cone over $(V,H)$ as in Theorem
\ref{counterexample}. Then $A_{\mathcal{X}/X}\geqslant 0$, but $X$
is not log canonical.
\end{corollary}

\begin{proof}
The corollary follows immediately from the fact that
$A_{\mathcal{X}/X}\geqslant 0$ is equivalent to
$\Vol_\text{BdFF}(X,0)=0$ (see \cite[Proposition 4.19]{BdFF} or
Corollary \ref{vol=lc}).
\end{proof}


\begin{thebibliography}{Lazarsfeld04}

\bibitem [BdFF12]{BdFF} S. Boucksom, T. de Fernex, C. Favre, \textit{The
volume of an isolated singularity}. Duke Math. J. \textbf{161}
(2012), 1455-1520.

\bibitem [BH12]{BH} A. Broustet, A. H\"{o}ring, \textit{Singularities of
varieties admitting an endomorphism}. Preprint, arXiv:1210.6254.

\bibitem [dFH09]{dFH} T. de Fernex, C. Hacon, \textit{Singularites
on normal varieties}. Compos. Math. \textbf{145} (2009), no. 2,
393-414.

\bibitem [Fulger11]{Ful} M. Fulger, \textit{Local volumes on normal algebraic
varieties}. Preprint, arXiv:1105.2981.

\bibitem [HMX12]{HMX12} C. Hacon, J. M$^\text{c}$Kernan, C. Xu,
\textit{ACC for log canonical thresholds}. Preprint,
arXiv:1208.4150.

\bibitem [Lazarsfeld04]{Laz} R. Lazarsfeld, \textit{Positivity in algebraic
geometry I}. Ergebnisse der Mathematik und ihrer Grenzgebiete. 3.
Folge. A Series of Modern Surveys in Mathematics [Results in
Mathematics and Related Areas. 3rd Series], vol. \textbf{48},
Springer-Verlag, Berlin, 2004.

\bibitem [KM98]{KM} J. Koll\'{a}r, S. Mori, \textit{Birational
geometry of algebraic varieties}. Cambridge Tracts in Mathematics,
\textbf{134}. Cambridge University Press, Cambridge, 1998.

\bibitem [Kollar92]{Kol} J. Koll\'{a}r et al., \textit{Flips and
abundance for algebraic threefolds}. Soci\'{e}t\'{e}
Math\'{e}matique de France, Paris, 1992. Papers from the Second
Summer Seminar on Algebraic Geometry held at University of Utah,
Salt Lake City, Utah, August 1991, Ast\'{e}risque, No. \textbf{211}
(1992).

\bibitem [OX12]{OX} F. Odaka, C. Xu, \textit{Log-canonical models of
singular pairs and its applications}. Math. Res. Lett.
\textbf{19}(2012), no.2, 325-334.

\bibitem [Shokurov00]{Sho} V. Shokurov, \textit{Complements on surfaces}.
Algebraic geometry, 10. J. Math. Sci. (New York) \textbf{102}
(2000), no. 2, 3876-3932.

\bibitem [Urbinati12]{Urb} S. Urbinati, \textit{Discrepancies of non-$\mathbb{Q}$-Gorenstein
varieties}. Michigan Math. J. Volume \textbf{61}, Issue 2 (2012),
265-277.

\end{thebibliography}
\end{document}